\documentclass[preprint,12pt,fleqn]{elsarticle}

\usepackage{amssymb}
\usepackage{amsmath}
\usepackage{amsthm}
\usepackage{algorithm}
\usepackage{caption}

\usepackage{amsfonts}
\usepackage{graphicx}
\usepackage{epstopdf}
\usepackage{algorithmic}
\ifpdf
  \DeclareGraphicsExtensions{.eps,.pdf,.png,.jpg}
\else
  \DeclareGraphicsExtensions{.eps}
\fi

\theoremstyle{remark}

\newtheorem{theorem}{Theorem}
\newtheorem{lemma}{Lemma}
\newtheorem{proposition}{Proposition}

\newtheorem{corollary}{Corollary}
\newtheorem{example}{Example}

\usepackage{microtype}
\usepackage{graphicx}
\usepackage{subfigure}
\usepackage{booktabs} 

\usepackage{hyperref}

\usepackage{color}
\usepackage{multirow}
\usepackage{nccmath, amssymb}
\usepackage{bbold}
\usepackage{thmtools,thm-restate} 

\usepackage{mathtools}
\usepackage{bm}
\usepackage{graphicx}
\usepackage{amsfonts}
\usepackage{amssymb}
\usepackage{amsmath}
\usepackage{newlfont}
\usepackage{xcolor}
\usepackage{caption}
\usepackage{dirtytalk}

\newcommand{\eps}{{\varepsilon}}

\newcommand{\Oc}{{\cal O}}
\newcommand{\Hc}{{\cal H}}
\newcommand{\R}{\mathbb R}
\newcommand{\PP}{\mathbb P}
\newcommand{\EE}{\mathbb E}
\newcommand{\C}{\mathbb C}

\newcommand{\N}{\mathbb N}

\newcommand{\ii}{\mathtt i}

\newcommand{\bigO}{\mathcal O}

\newcommand{\Amx}{A=[a_{ij}]\in{\mathbb C}^{d,d}}
\newcommand{\Amxx}{A=[a_{ij}]\in{\mathbb C}^{m,d}}

\DeclareMathOperator*{\range}{\ensuremath{\text{\rm Im}}}

\DeclareMathOperator{\Diag}{\ensuremath{\text{\rm diag}}}

\providecommand{\norm}[1]{\lVert#1\rVert}

\providecommand{\abs}[1]{\lvert#1\rvert}

\newcommand{\Spec}[1]{\Lambda_{#1}}
\newcommand{\PSpec}{\Lambda_{\varepsilon}}
\newcommand{\dti}{\delta}
\newcommand{\srad}{\rho}
\newcommand{\pcon}{p}
\newcommand{\sumcon}{s}
\newcommand{\kreiss}{\kappa}
\newcommand{\spH}{\mathcal{H}}
\newcommand{\Lii}{\mathcal{L}^2_\pi}
\newcommand{\one}{\mathbb 1}
\newcommand{\cf}{\one_\pi}
\newcommand{\dt}[1]{q_{#1}} 
\newcommand{\Koop}{{E}_{\pi}}  %Koopman  on L2
\newcommand{\cKoop}{T_{\pi}}  %Deflated Koopman  on L2
\newcommand{\adjKoop}{E^{*}_{\pi}}  %adjoint Koopman  on L2
\newcommand{\adjcKoop}{T^{*}_{\pi}}  % adjoint defleated Koopman on L2
\newcommand{\Estim}{G}  %Estimator
\newcommand{\EEstim}{\widehat{\Estim}}  %Empirical estimator of the Koopman 
\newcommand{\lrm}{\mathcal M}
\newcommand{\ECx}{\widehat{C}} %Empirical Covariance of inputs
\newcommand{\ECxy}{\widehat{C}_{+}}  %Emp Cross-covariance
\newcommand{\reg}{\gamma} 

\newcommand{\ECreg}{\widehat{C}_\reg}
\newcommand{\Kx}{K_x} % kernel of inputs
\newcommand{\CKx}{\bar{K}_x} % kernel of inputs
\newcommand{\TS}{S_\im}  % Cannonical injection
\newcommand{\ES}{\widehat{S}_x} % Sampling of inputs
\newcommand{\im}{\pi}

\usepackage{amsopn}

\journal{Journal of Computational and Applied  Mathematics}

\begin{document}

\begin{frontmatter}

\title{Scalable Pseudospectral Analysis via Low-Rank Approximations of Dynamical Systems}

\author[pmf,iit]{V.~R.~Kosti\'c}
\ead{vladimir.slk@gmail.com}

\author[pmf]{Lj.~Cvetkovi\'c}
\ead{lila@dmi.uns.ac.rs}

\author[ftn]{D.~Lj.~Cvetkovi\'c}
\ead{draganacvet@uns.ac.rs}

\address[pmf]{Faculty of Science, University of Novi Sad,
Trg Dositeja Obradovi\'ca 4,
21000 Novi Sad,
Serbia}

\address[ftn]{Faculty of Technical Sciences, University of Novi Sad,
Trg Dositeja Obradovi\'ca 6,
21000 Novi Sad,
Serbia}

\address[iit]{Instituto Italiano di Tecnologia,
Via Melen 83,
16152 Genoa,
Italy}

%% Abstract
\begin{abstract}
Pseudospectral analysis is fundamental for quantifying the sensitivity and transient behavior of nonnormal matrices, yet its computational cost scales cubically with dimension, rendering it prohibitive for large-scale systems. While existing research on scalable pseudospectral computation has focused on exploiting sparsity structures, common in discretizations of differential operators, these approaches are ill-suited for machine learning and data-driven dynamical systems, where operators are typically dense but approximately low-rank. In this paper, we develop a comprehensive low-rank framework that dramatically reduces this computational burden. Our core theoretical contribution is an exact characterization of the pseudospectrum of arbitrary low-rank matrices, reducing the evaluation of resolvent norms to eigenvalue problems of dimension proportional to the rank. Building on this foundation, we derive rigorous inclusion sets for the pseudospectra of general matrices via truncated and randomized low-rank approximations, with explicit perturbation bounds. These results enable efficient estimators for key stability quantities, including distance to instability and Kreiss constants, at a cost that scales with the effective rank rather than the ambient dimension. We further demonstrate how our framework naturally extends to data-driven settings, providing pseudospectral analysis of transfer operators learned from nonlinear and stochastic dynamical systems. Numerical experiments confirm orders-of-magnitude speedups while preserving accuracy, opening pseudospectral analysis to previously intractable high-dimensional problems in computational PDEs, control theory, and data-driven dynamics.
\end{abstract}

%%Graphical abstract
%\begin{graphicalabstract}
%\includegraphics{grabs}
%\end{graphicalabstract}

%%Research highlights
%\begin{highlights}
%\item Research highlight 1
%\item Research highlight 2
%\end{highlights}

%% Keywords

\end{frontmatter}

\section{Introduction}

Pseudospectral analysis is a cornerstone of modern dynamical systems theory, offering crucial insight into the transient behavior, stability robustness, and sensitivity to perturbations of linear operators~\cite{TE2020}. Unlike the spectrum alone, pseudospectra capture the often-dramatic effects of nonnormality, a property ubiquitous in discretized differential operators, adjoint systems, and data-driven models of complex dynamics~\cite{schmid2007}. Although pseudospectral theory is classically formulated for linear operators, its relevance for nonlinear deterministic and stochastic dynamical systems has become increasingly apparent, \cite{Colbrook2021}. Nonlinear systems, even when linearized around stable equilibria or attractors, may exhibit pronounced transient growth that cannot be inferred from the spectrum of the linearized operator alone, \cite{TE2020}. This effect is further amplified in the presence of noise, where the dynamics is naturally modeled by stochastic differential equations (SDEs), \cite{pavliotis2014}. In this setting, the relevant linear objects are the infinitesimal generators of diffusion processes and the associated Koopman and Perron–Frobenius operators, which arise from nonlinear dynamics but act linearly on spaces of observables or probability densities, \cite{Lasota1994, Brunton2022}. These operators can be highly non-normal, leading to substantial pseudospectral spreading and explaining the occurrence of strong transient amplification and reduced robustness of stability, even when the system is asymptotically stable in a mean-square or almost-sure sense, \cite{Kostic-ICML2024, ELFALLAH2002135}. Consequently, pseudospectral analysis provides a natural framework for quantifying robustness and short-term predictability in nonlinear and stochastic dynamics, \cite{Colbrook2021, Kostic-ICLR2024}. However, despite its conceptual power, pseudospectral computation remains prohibitively expensive for large-scale systems, limiting its practical utility in applied and computational settings.

\smallskip

\noindent\textbf{The computational barrier.~}
For a $d \times d$ matrix $A$, computing the $\varepsilon$-pseudospectrum requires repeated evaluation of the resolvent norm $\|(zI - A)^{-1}\|$ over a grid in the complex plane. Each evaluation involves a singular value decomposition or the solution of a large linear system, leading to an $\mathcal{O}(d^3)$ cost per point~\cite{TE2020}. In continuous-time settings, where stability is determined by pseudospectral contours in the left half-plane, the computational burden is further amplified. Key quantities such as the \emph{distance to instability} and the \emph{Kreiss constant}, the object essential in control theory and numerical analysis, depend on global nonconvex pseudospectral optimization and inherit the same scalability issues~\cite{overton2001}.  For discretized differential operators arising in finite element or spectral methods, the dimension $d$ can easily reach thousands or millions, rendering traditional algorithms entirely impractical. This was the reason why algorithms for computing the $\varepsilon$-pseudospectrum, distance to instability, and Kreiss constant have matured from dense, grid-based methods to large-scale iterative approaches that exploit sparsity, see recent monograph \cite{GLmono} and references within. In contrast, effective low-rank structure  of matrices and linear operators modeling dynamical systems has been overlooked.

\noindent\textbf{Low rank structures.~}
 Interestingly, in many scientific and engineering applications, the relevant operators admit accurate low-rank representations, either by design (e.g., via model reduction) or as a consequence of fast singular value decay (e.g., in smooth or high-dimensional data)~\cite{Martinsson2020, udell2019}. In computational statistics and machine learning, the notion of effective rank is fundamental in separating the signal, related to the leading singular values,  by a gap from the noise in data, related to small singular values, \cite{RoyV}.  In the analysis of nonlinear and stochastic systems, transfer operators such as the Koopman and Perron--Frobenius operators that act linearly on spaces of observables or densities are intrinsically infinite-dimensional and can be highly nonnormal \cite{Lasota1994, Brunton2022}. 
Learning Koopman and transfer operators using machine learning approaches, such as dictionary-based methods \cite{Kutz2016, Colbrook2019}, neural networks \cite{Lusch2018,
Kostic-ICLR2024}, and Galerkin projections with data-adaptive bases \cite{Kostic2022, Bevanda2023,Bevanda2021} yields interpretable low-rank models of high-dimensional, nonlinear systems. When integrated with tools like randomized \cite{turri2026randomized} and neural \cite{Kostic2024-NCP, jeongefficient} SVD, they provide scalable frameworks for analyzing transient dynamics and spectral behavior. Understanding the pseudospectral properties of these approximations is essential for predicting transient growth, stability margins, and robustness in data-driven models of complex dynamics.

\smallskip

\noindent\textbf{The fundamental gap.~}
Despite the central role of pseudospectra in dynamics and the rise of low-rank methods in machine learning of dynamical systems, a critical theoretical and practical gap persisted: \textit{How to efficiently exploit approximately low rank structure when determining pseudospectral properties?} More precisely: Can the pseudospectrum of a large matrix be cheaply and reliably estimated from a low-rank surrogate? Further, what is the \emph{minimal} rank required to capture key pseudospectral features, such as $\varepsilon$-contours, distance to instability, or the Kreiss constant, within a prescribed tolerance? How do these relationships translate to the transient dynamics of the underlying system, both linear and nonlinear? Existing literature provides limited answers, and no unified framework exists to bridge low-rank approximation theory with efficient pseudospectral analysis for large-scale and data-driven dynamical systems.

\smallskip

\noindent\textbf{Contributions of this work.~}
This paper fills this gap by developing a comprehensive theory of \emph{low-rank pseudospectral estimation} and demonstrating its utility in computational and applied settings. Our key contributions are: (i) \underline{\textit{Characterization of pseudospectra:}} We derive explicit and efficiently computable expressions for the pseudospectra of arbitrary low-rank matrices. These formulas avoid expensive resolvent norm evaluations and reduce the problem to computations of smallest eigenvalues of structured matrices whose dimension depends only on the rank; (ii) \underline{\textit{Low-rank intersection theorems:}} For key pseudospectral tasks, such as computing the distance to instability or estimating Kreiss constants, we provide reduced eigencharacterizations of the vertical line and circle intersections that operate only on low-dimensional subspaces; (iii) \underline{\textit{Randomized algorithms for large-scale estimation:}} We combine our theoretical characterizations with randomized SVD and sketching techniques to produce scalable algorithms that approximate pseudospectra, distance to instability, and Kreiss constants for large matrices at substantially reduced cost; (iv) \underline{\textit{Extensions to nonlinear and stochastic dynamics:}} We show how our framework integrates naturally with data-driven operator approximation methods (e.g., DMD, Koopman models) to analyze transient dynamics and robustness in nonlinear deterministic and stochastic systems. This provides a rigorous foundation for pseudospectral analysis in settings where only trajectory data are available; (v) \underline{\textit{Numerical validation and applications:}} We illustrate the efficiency and accuracy of our methods on  problems from data-driven models, demonstrating orders-of-magnitude speedups while retaining essential dynamical information.

\smallskip

\noindent\textbf{Paper organization.~} The paper is organized as follows. In Section 2, we introduce the problem setting and review the necessary background on pseudospectra, transient dynamics, and transfer operators for linear and stochastic dynamical systems. Section 3 presents a new explicit characterization of the pseudospectrum of low-rank matrices, together with eigen-characterizations for the problem of computing intersections with lines and circles. In Section 4, we extend these results to the pseudospectral estimation of general matrices via truncated and randomized low-rank approximations, and provide rigorous inclusion guarantees. Section 5 illustrates the applicability of the proposed framework to nonlinear stochastic dynamical systems through data-driven approximations of transfer operators. Finally, Section 6 concludes the paper with remarks and perspectives for future research.

\section{The problem, background and preliminaries}

In this work we study the problem of efficiently estimating the impact of non-normality in general systems. While for the ease of presentation we focus on the discrete time case, the continuous time one will be explicitly mentioned whenever adaptation of the results is not straightforward. For the linear dynamics $x_t= A x_{t-1} = A^t x_0$, $t\in\N$, the main tool in understanding this phenomena is the notion of pseudospectrum \cite{TE2020}. More precisely, 
for arbitrary bounded linear operator $A\colon \Hc\to\Hc$ on a separable Hilbert space $\Hc$ and arbitrary $\varepsilon >0$,  the $\varepsilon$-pseudospectrum of  $A$ is given by 
\begin{equation} \label{def:pseudo} 
\Spec{\varepsilon}(A) := \{ z \in \C \; :  \|(A-z I)^{-1} \|^{-1} \leq \varepsilon \} = \bigcup_{\|E\| \leq\eps}\Spec{0}(A+E),
\end{equation}
where $\norm{\cdot}:=\sup_{x\in\Hc}\norm{(\cdot)x}/\norm{x}$ is in the induced operator norm, and using the convention that $\|A^{-1} \|^{-1} = 0$ when $A$ does not have a bounded inverse. Note that $\Spec{}(A)\equiv\Spec{0}(A)$ then denotes the spectrum of $A$, and whenever $\Hc=\C^d$ and $\Amx$, spectrum is the set of $A$'s eigenvalues $ \Spec{}(A) := \{ \lambda \in \C  : \exists \ x  \in \C^d \setminus \{0\} ,  \; \mbox{such that} \; Ax=\lambda x \} $ and $\varepsilon$-pseudospectrum coincides with the set of all eigenvalues of matrices that are \emph{$\varepsilon -$close} to $A$. 

Therefore, the pseudospectrum reveals how sensitive are the eigenvalues to perturbations, and, hence, numerical computation and/or statistical estimation. This crucially depends on \textit{normality} of $A$. Namely, if $A$ is normal operator, that is $AA^*=A^*A$,  then $\Spec{\varepsilon}(A)=\Spec{}(A)+\{z\in\C\,\vert\,\abs{z}\leq\varepsilon\}$ and the problem of determining spectrum of $A$ is well conditioned. On the other hand, when $A$ is a non-normal operator, $\Spec{\varepsilon}(A)$ can grow even exponentially w.r.t. $\varepsilon$ implying that computation/estimation of eigenvalues can be unreliable. In particular, the operator can be asymptotically stable, that is $\lim_{t\to\infty}\norm{A^t}=0$ or, equivalently, its spectral radius is $\srad(A):=\max\{\abs{\lambda}\,\vert\,\lambda\in\Spec{}(A)\}<1$, but under small perturbation it may become unstable. This is reflected in the notion of the \textit{distance to instability} of $A$, 
\begin{equation}\label{eq:d2i}
\dti(A):= \min\{\eps\geq0\,\colon\, \exists z\in \Spec{\varepsilon}(A) \text{ s.t. } \abs{z} = 1\}  = \inf_{z\in\C,\,\abs{z}=1}\norm{(A\!-\! zI)^{-1}}^{-1}
\end{equation}
that measures the distance of the operator's spectra to the unit circle relative to its sensitivity to perturbations, which for normal operators equals $1\!-\!\srad(A)$. 

\paragraph{Transient dynamics and robustness of stability.} Related to linear dynamics, for an asymptotically stable bounded linear operator $A$, depending on the \textit{normality} the convergence of powers to zero might not be monotone. Namely, if $AA^*=A^*A$, then $\norm{A^t} = [\srad(A)]^t$, and consequently, 
\begin{equation}\label{eq:constants}
\pcon(A):=\sup_{t\in\N_0}\norm{A^t}=1\quad\text{ and }\quad \sumcon(A):=\sum_{t=0}^{\infty}\norm{A^t} = \frac{1}{1-\srad(A)}.
\end{equation}
Otherwise, the sequence $(\norm{A^t})_{t\in\N_0}$ may exhibit a \textit{transient growth} before converging to zero, that can be estimated, c.f. \cite{TE2020, ELFALLAH2002135}, by
\begin{equation}\label{eq:kreiss_bound}
    \kreiss(A){\leq} \pcon(A){\leq} ed\, \kreiss(A),\text{ when }\Hc{=}\C^d \; \text{or, in general,} \; \kreiss(A){\leq} \pcon(A){\leq} (e/2) [\kreiss(A)]^2,
\end{equation}
where $\kreiss(A)$ is the \textit{the Kreiss} constant of $A$ defined as 
\begin{equation}\label{eq:kreiss}
\kreiss(A):= \sup_{\eps\geq0}\frac{\srad_{\eps}(A)-1}{\eps} =\sup_{z\in\C,\,\abs{z}>1}(\abs{z}\!-\!1)\norm{(A\!-\! zI)^{-1}} \geq1,
\end{equation} 
and $\srad_{\varepsilon}(A):=\max\{\abs{z}\,\colon\,z\in\Spec{\eps}(A)\}$ is the pseudospectral radius, see e.g. \cite{Kressner}.

For highly non-normal operators $\kreiss(A)\gg 1$, indicating a large transient growth, which is also related to much smaller distance to instability $\dti(A)\ll 1\!-\!\srad(A)$, and larger cumulative effect $\sumcon(A) \gg 1/(1\!-\!\srad(A))$. Nevertheless, the latter quantity always remains bounded, since due to $\limsup_{t\rightarrow \infty}\norm{A^t}^{1/t} =  \rho(A)<1$, there exists the smallest integer $\ell$ such that $\norm{A^\ell}<1$, and, consequently, $$\sumcon(A)\leq \tfrac{1}{1\!-\!\norm{A^\ell}}\tfrac{\norm{A}^\ell\!-\!1}{\norm{A}\!-\!1}<\infty.$$

In the context of continuous time dynamics $dx_t = Ax_t$, $t\geq0$, we have similar effect, but instead of matrix powers, the object is the matrix exponential $e^{tA}$ for which the the bounds \eqref{eq:kreiss_bound} hold for the continuous versions of \eqref{eq:constants} and \eqref{eq:kreiss} given by $p_c(A) = \sup_{t\geq0}\|e^{tA}\|$ and 
\begin{equation}\label{eq:kreiss_c}
\kreiss_c(A):= \sup_{\eps\geq0}\frac{\alpha_{\eps}(A)}{\eps} =\sup_{z\in\C,\,\Re({z})>0}\Re({z})\norm{(A\!-\! zI)^{-1}} \geq1,
\end{equation} 
where $\alpha_{\varepsilon}(A):=\max\{\Re({z})\,\colon\,z\in\Spec{\eps}(A)\}$ is the pseudospectral abscissa.

\paragraph{Transfer operators of ergodic processes.} We conclude this section by presenting the pivotal role of pseudospectral theory in data-driven methods for, in general nonlinear and stochastic, dynamical systems based on transfer operator theory \cite{Lasota1994}, as recently demonstrated in \cite{Kostic2023, Kostic-ICML2024}. To this end, let  $(X_t)_{t \in \N_0}$ be time-homogeneous Markovian system, where the state at time $t\in \N$ is an $\R^d$-valued  random variable $X_t$ with probability law {$\pi_t$}, that is  $\PP[X_{t+1}\,\vert\,(X_s)_{s=0}^t] = \PP[X_{t+1}\,\vert\,X_t]$ is independent of $t$. Further, assume that the dynamical system is {\em geometrically ergodic}, meaning that there exists a \textit{unique} probability distribution $\pi$, called \textit{invariant measure}, such that $X_0\sim\pi$ implies $X_t\sim\pi$, for every $t\in\N$, so that distributions {$(\pi_t)_{t\in\N_0}$} converges strongly to $\pi$ with a geometric rate. Such dynamical systems are general enough to capture several important phenomena, including (discretized) Langevin dynamics for atomistic simulations \cite{Davidchack2015}, or other systems constructed from the discretization of stochastic differential equations with diverse application in finance, climate modeling, etc. \cite{MarkovChainsFinance1992, CalvetFisher2008, Liu2020WeatherDerivatives}. They can be studied via Markov operators, and, in particular with \emph{backward transfer operators} $\Koop\colon\Lii\to\Lii$ defined on the space $\Lii\equiv\Lii(\R^d)$ formed by  square integrable functions on $\R^d$ w.r.t. the invariant measure as
\begin{equation}\label{eq:transfer-operators}
[\Koop f](x) := \EE[ f(X_{t+1})\,\vert X_{t} = x],\quad x\in\R^d,\, t\in\N_0. 
\end{equation}   
Due to their prominence in the data-driven (deterministic) dynamical systems community, see e.g. \cite{Brunton2022, Colbrook2021}, we also call $\Koop$ the (stochastic) Koopman operators. The significance of these operators lies in their ability to \textit{globaly linearize} the underlying Markov processes. Namely, for every observable $f\in\Lii$, computing its expected value after $t$ time steps from some initial state $x\in\R^d$ is simply powering of Koopman operator $\Koop$, i.e. $\EE[ f(X_t)\,\vert\,X_0 = x] = [\Koop^t f](x)$. On the other hand,  using duality between observables and state distributions \cite{Lasota1994}, if $\pi_0$ is absolutely continuous w.r.t. the invariant measure $\pi$, that is, it has a density $\dt{0}:=d\pi_{0} / d\pi \in L^1_\pi(\R^d)$ defined via the Radon-Nikodyn derivative, and additionally the density is square-integrable, then, for every $t\in\N$ one has $\dt{t}:=d\pi_{t} / d\pi \in \Lii$, and the flow of the probability distributions $(q_t)_{t \in \N}$ follows the  \textit{linear dynamic} in the space $\Lii$, given by the equation $ q_t =  (\Koop^*)^t q_0$, $t \in \N_0$. Among all observables/densities, constant ones play a particular role. Namely, \eqref{eq:transfer-operators} implies that $\Koop\cf = \adjKoop\cf = \cf$, where $\cf\in\Lii$ is the function $\pi$-almost everywhere equal to $1$, and, since the process is geometrically ergodic, the (largest) eigenvalue 1 is unique, and $\lim_{t\to\infty}\norm{\dt{t}-\cf}=0$. So, deflating the largest eigenvalue from $\cKoop:=\Koop - \cf\otimes\cf$, with $\otimes$ being outer product in $\Lii$, we obtain that $\srad(\cKoop)<1$ and $\dt{t}\!-\!\cf = \adjcKoop(\dt{t-1}\!-\!\cf) = (\adjcKoop)^t(\dt{0}-\cf)$ becomes an asymptotically stable linear dynamical system in $\Lii$ that can be studied with pseudospectral tools.

\paragraph{Machine learning methods for transfer operators.} In recent years, there has been a growing interest on data-driven dynamical systems where instead of the classical modeling by differential equations derived the first principles and parameter fitting, one aims to learn the model based on collected data with no, or some partial, a  priori knowledge of the dynamics. In this setting $\Koop$, and hence $\cKoop$, is not known, and a key challenge is to learn it from the observed dynamics.  An appealing class of operator regression learning algorithms \cite{Brunton2022, Kostic2022, Kutz2016, Bevanda2023, tropp} aims to estimate the Koopman operator on a predefined hypothesis space $\Hc$ consisting of functions from $\Lii$. Recently developed statistical learning theory for this problem in \cite{Kostic2023, Kostic-ICML2024} showed that the low-rank empirical estimators $\Estim\colon\Hc\to\Hc$ are able to properly approximate $\Koop$ and achieve consistent long term forecasting of distributions {$(\pi_t)_{t\in\N_0}$} by learning $\cKoop$ whenever $\Hc$ is (an infinite-dimensional) reproducing kernel Hilbert space (RKHS) defined by a universal symmetric and positive definite kernel function  
$k:\R^d\times\R^d \to \R$ \cite{aron1950, Steinwart2008}. Alas, the analysis exposed that for finite number of observed samples {${\cal D}_n=(x_i)_{0\leq i\leq n}$,} the forecasting error bounds come with the constant $\min\{\pcon(\cKoop)\sumcon(\Estim), \pcon(\Estim)\sumcon(\cKoop)\}$ motivating the necessity of efficient computation of the pseudospectral objects of low rank operator $\Estim\colon\Hc\to\Hc$ and statistical estimation of the pseudospectral objects of $\cKoop\colon\Lii\to\Lii$.
Unfortunately, classical finite element discretizations for the numerical approximation of the Koopman operator associated with stochastic differential equations face severe limitations due to the curse of dimensionality. Resolving the sharp peaks of the resolvent norm inherent to pseudospectral computations necessitates extremely fine spatial discretizations, further amplifying this dimensionality bottleneck. As a result, classical FEM-based spectral and pseudospectral analyses of Koopman operators for SDEs become intractable even for moderately high-dimensional systems, leaving the data driven method practically the only option.
 
\section{Pseudospectrum of low-rank matrices}\label{sec:lowrank}

In this section we derive computationally appealing charaterization of the pseudospectrum of a general square low rank matrix $A=UV^*$, where $U,V\in\mathbb{C}^{d,r}$. While computing the spectrum of low-rank $A$ is well known result, see e.g. \cite{SS1990}, the analog for pseudospectrum remained, up no our knowledge, an open problem. The core idea of our approach is to replace the computation of the smallest singular value of a large $A$ ($d\gg r$), by the computation of the smallest eigenvalue of small $2r \times 2r$ matrix. This key result is as follows.

\begin{theorem}\label{thm:main}
Given $d>r\geq1$, and any $U,V\in\mathbb{C}^{d,r}$, let 
\begin{equation}\label{eq:lrm}
\lrm_{U,V}(z):=\left[
\begin{array}{cc}
|z|^2I -z\,U^*V & (|z|^2I- z\,U^*V)\,U^*U \vspace*{0.3cm}\\ 
V^*V & |z|^2I -\overline{z}\, V^*U + V^*V\,U^*U 
  \vspace*{0.1cm}
\end{array}
\right]\in \C^{2r,2r}.
\end{equation}
Then, for all  $z\in\C$, $\Spec{}(\lrm_{U,V}(z))\subset\R_{+}$ and $$\mu_{U,V}(z):=\sqrt{\lambda_{min} \big(\lrm_{U,V}(z)\big)} = \sigma_{\min}(z I- UV^*),$$ and, consequentially, the following characterization of the $\eps$-pseudospectrum of $A=UV^*$ holds true:
$\PSpec(A) = \big\{z \in \C \; :
\mu_{U,V}(z) \leq  \eps  \big\}
$.
\end{theorem}
{\em Proof:}  
As a first step, we prove the claim under the assumption that  $U^*U=I$.

By $\hat{U} \in \mathbb{C}^{d,d}$ denote an unitary matrix, which first $r$ columns are columns of $U$. Also, by $\hat{V} \in \mathbb{C}^{d,d}$ denote the matrix, which first $r$ columns are the ones of $V$, while the remaining $d-r$ arxe zero, i.e.
$ \hat{U}=\left[
\begin{array}{c|c}
U  & W 
\end{array}
\right] , \;\; \hat{V}=\left[
\begin{array}{c|c}
V  & 0 
\end{array}
\right].
$
Then
$UV^*=\hat{U} \hat{V}^* ,$
so the resolvent of $UV^*$ can be expressed as:
\[
\| (UV^*-zI)^{-1}\|^{-2}= 
\| (\hat{U} \hat{V}^* -zI)^{-1}\|^{-2}= 
\min_{x\in \mathbb{C}^d \setminus \{ 0\}}  \frac{\|(\hat{U} \hat{V}^*-zI)x\|^2}{\|x\|^2}=\]\[
=\min_{x\in \mathbb{C}^d \setminus \{ 0\}} \frac{\|(\hat{V}^*  -z\hat{U}^*)x\|^2}{\|x\|^2} .
\]
For arbitrary $x \in \mathbb{C}^{d}$, 
$
\big(\hat{V}^*-z\hat{U}^* \big)x=\left[
\begin{array}{c}
\big(V^*-zU^*\big)x\\ 
-z W^*x
\end{array}
\right],$ and $$ \|(\hat{V}^*  -z\hat{U}^*)x\|^2=\|(V^*-z U^*)x\|^2+|z|^2\|W^*x\|^2 .$$
Since
$$
\|x\|^2=\|\hat{U}^*x\|^2=\left\| \left[
\begin{array}{c|c}
U &
W
\end{array}
\right]^* x \right\|^2= \|U^*x\|^2+\|W^*x\|^2 ,
$$
we have 
$$
\frac{\left\|(\hat{V}^*  -z\hat{U}^*)x\right\|^2}{\| x\|^2}=\frac{\|(V^* -z U^*)x\|^2+|z|^2\Big( \| x\|^2- \|U^*x\|^2\Big)}{\| x\|^2}=
$$$$
=|z|^2+\frac{\|(V^* -z U^*)x\|^2 -|z|^2\|U^*x\|^2}{\| x\|^2} .
$$
The numerator in the above fraction is
$$
\|(V^* -zU^*)x\|^2 -|z|^2\|U^*x\|^2 =x^*(V-\overline{z}U) (V^* -zU^*)x -|z|^2x^*UU^*x=
$$$$=x^*\big(VV^*-\overline{z}UV^*  -zVU^* +\overline{z}zUU^*\big)x -|z|^2x^*UU^*x=
$$$$=x^*\big(VV^*-\overline{z}UV^*  -zVU^* \big)x,
$$
hence 
$$
\min_{x\in \mathbb{C}^d \setminus \{ 0\}} \frac{\|(\hat{V}^*  -z\hat{U} ^*)x\|^2}{\|x\|^2} =|z|^2+\lambda_{min}\big(VV^*-\overline{z}UV^*  -zVU^* \big) ,
$$
where $\lambda_{min}$ denotes the minimal (real) eigenvalue of the Hermitian matrix $VV^*-\overline{z}UV^*  -zVU^*$. We will show that the nonzero eigenvalues of this matrix are the same as the nonzero eigenvalues of $\lrm_{U,V}(z)$. In order to do that, we will represent the matrix $VV^*-\overline{z}UV^*  -zVU^*$ as a product of two matrices, which dimensions are $d \times 2r$ and $2r \times d$, respectively:
\begin{equation}\label{eq:large_mx}
VV^*-\overline{z}UV^*  -zVU^*=\left[
\begin{array}{c|c}
V  & V -\overline{z}U
\end{array}
\right]\left[
\begin{array}{c}
- zU^*\\ \hline
V^*
\end{array}
\right]  .
\end{equation}
As  nonzero eigenvalues remain the same if we reverse multiplication, we conclude that matrix 
\begin{equation}
\left[
\begin{array}{c}
- zU^*\\ \hline
V^*
\end{array}
\right]\left[
\begin{array}{c|c}
V  & V -\overline{z}U
\end{array}
\right] \hspace{-0.1cm} = \hspace{-0.1cm}
\left[
\begin{array}{cc}
 -zU^*V & |z|^2I- zU^*V \vspace*{0.3cm}\\ 
V^*V &  -\overline{z}V^*U + V^*V 
  \vspace*{0.1cm}
\end{array}
\right]\hspace{-0.1cm} =\hspace{-0.1cm} \lrm_{U,V}(z)- 
 |z|^2I 
\label{velikamatrica}
\end{equation}
has the same nonzero eigenvalues as \eqref{eq:large_mx}, which are all real. 
Consequently,
\begin{equation} \label{nula}\lambda_{min}\big(VV^*-\overline{z}UV^*  -zVU^* \big)= \min \Big\{ 0, \lambda_{min}(\lrm_{U,V}(z))-|z|^2\Big\}.
\end{equation}
We will show that
$\lambda_{\min}\lrm_{U,V}(z) - |z|^2 \le 0$.  Denote, for a moment 
$ B:=V^*V  $ and $  C:=-z\,U^*V .$
Then
$$\lrm_{U,V}(z)- |z|^2I=\left[
\begin{array}{cc}
C & |z|^2I+ C \vspace*{0.3cm}\\ 
B & C^* +B
  \vspace*{0.1cm}
\end{array}
\right] \quad \text{and}$$
$$
\det\Big(\big(\lrm_{U,V}(z)-|z|^2I\big)-\lambda I\Big)=
\det
\left[
\begin{array}{cc}
C-\lambda I & |z|^2I +C\vspace*{0.3cm}\\ 
B &  C^*+B-\lambda I
  \vspace*{0.1cm}
\end{array}
\right]=$$$$=
\det
\left[
\begin{array}{cc}
C-\lambda I & |z|^2I+\lambda I \vspace*{0.3cm}\\ 
B &  C^*-\lambda I
  \vspace*{0.1cm}
\end{array}
\right]
$$
Since the upper right block is a scalar matrix, it holds that
$$
p(\lambda ):=\det\Big(\big(\lrm_{U,V}(z)-|z|^2I\big)-\lambda I\Big)=
\det\Big(
(C^*-\lambda I)(C-\lambda I)-
(|z|^2+\lambda)B\Big)
$$
Consider  matrix 
$F(\lambda):=(C^*-\lambda I)(C-\lambda I)-
(|z|^2+\lambda)B.$
For real $\lambda$ it is a Hermitian matrix. On one hand, 
$$F(0)= |z|^2V^*UU^*V-|z|^2V^*V = |z|^2V^*\big(UU^*-I\big)V 
$$
is negative semidefinite, since 
$\big(I -UU^*\big)$ is an orthogonal projector and hence positive semidefinite.
On the other hand, 
$$F(-|z|^2)= (B^*+ |z|^2I)(B+|z|^2 I)
$$
is positive semidefinite.  If either $F(-|z|^2)$ or $F(0)$  is singular, then $\lambda_{\min}\lrm_{U,V}(z) - |z|^2 \le 0$ follows immediately. Assume therefore that both matrices are nonsingular. Then,  
$F(-|z|^2)$ is positive definite and $F(0)$ is negative definite. Let $\mu_{\min}(\lambda)$ denotes the minimal eigenvalue of the hermitian $F(\lambda)$. Recalling, the classical result on Hermitian matrix families,  $\mu_{\min}(\lambda)$
 is a real-valued continuous function of $\lambda$. Since
$\mu_{\min}(-|z|^2)>0 $ and $ \mu_{\min}(0)<0$,
the intermediate value theorem implies that $\lambda
\in \big( -|z|^2, 0\big)$ such that $\mu_{\min}(\lambda)=0$. Hence, $F(\lambda)$ is singular and 
$\det F(\lambda) = 0$. 
This proves that the polynomial
$p(\lambda)=\det F(\lambda)$
has at least one real zero in the interval $\big[ -|z|^2, 0\big]$, i.e. matrix $\lrm_{U,V}(z)-|z|^2I$ has at least one nonpositive eigenvalue.
Hence,
 $$
|z|^2+\lambda_{min}\big(VV^*-\overline{z}UV^*  -zVU^* \big) = \lambda_{min}\left(\lrm_{U,V}(z)\right) ,
$$
and
$$
\| (UV^*-zI)^{-1}\|^{-2}= \lambda_{min}\left(\lrm_{U,V}(z)\right) = [\mu_{U,V}(z)]^2 .
$$

\medskip
Now, let us consider the general case. Let $U\in\C^{d\times r}$ be arbitrary. Define $(U_t)_{t>0}$ family of full rank matrices such that $\lim_{t\to 0}U_t=U$. Then, obviously, we can write $U_t (U_t^* U_t)^{-1/2} (U_t^* U_t)^{1/2} V^*$ and apply the proven claim on $$U\leftarrow U_t (U_t^* U_t)^{-1/2} \quad \text{and}\quad V\leftarrow(U_t^* U_t)^{1/2} V^*.$$ However, it is easy to see that $\lrm_{U_t (U_t^* U_t)^{-1/2},V(U_t^* U_t)^{1/2}}(z)$ is similar to $\lrm_{U_t ,V}(z)$
by applying the block diagonal scaling $\Diag [(U_t^* U_t)^{1/2},(U_t^* U_t)^{-1/2}]$. So, using the continuity of the eigenvalues, and letting $t\to0$, we have proven that $\sigma_{\min}(zI-A) = \mu_{U,V}(z)$. 

Finally, to conclude the proof, according to  (\ref{def:pseudo}),
the $\eps$-pseudospectrum of $UV^*$ is $
\Lambda_{\varepsilon}(UV^*)
 = \Big\{ z \in \C \; :  \mu_{U,V}(z) \leq \varepsilon\Big\}$.
$\Box$ 

Next, we give an alternative characterization of $\mu_{U,V}$ via the quadratic eigenvalue problem.

\begin{corollary}
Under the assumptions of Theorem \ref{thm:main}, $\lambda$ is the eigenvalue of $\lrm_{U,V}(z)$, with corresponding eigenvector
 $x = [x_1^*\,\vert\,x_2^*]^*$ if and only if  $\lambda$ satisfies the generalized Hermitian eigenvalue problem
\begin{equation}\label{eq:gep}
\text{(GEP)} \;\;
\left(\lambda\,\begin{bmatrix}
0 & I \\[2pt]
I  & - U^*U 
\end{bmatrix} \hspace{-0.1cm}
 - \hspace{-0.1cm}
\begin{bmatrix}
V^*V & |z|^2-\overline{z}\,V^*U \\[2pt]
|z|^2- z\,U^*V & 0
\end{bmatrix}\right)\hspace{-0.1cm}
\begin{bmatrix} x_1+U^*U x_2 \\ x_2 \end{bmatrix} =  0.
\end{equation}
Furthermore, if  $U^*U = I$,  for every $z \notin \Lambda(V^*U)$,  $\lambda$ is the eigenvalue of $\lrm_{U,V}(z)$ if and only if $\eta:=\lambda -|z|^2$  is the eigenvalue of the quadratic eigenvalue problem
\begin{equation}\label{eq:qep}
\text{(QEP)}\quad\det\Big(\eta^{2} \,I + \eta\,(z U^* V + \bar{z} V^* U\!-\!V^*V) + |z|^{2}V^* (UU^*\!-\!I)V\Big)=0.
\end{equation}
\end{corollary}

{\em Proof:} 
Denoting $G = U^*U$, $B = V^*V$ and $F = |z|^2 I - zU^*V $, the eigenvalue problem $\lrm_{U,V}(z) x = \lambda x$ is equivalent to two block equations 
\begin{align*}
Fx_1+ F G x_2&= \lambda  x_1, \\
B x_1 +(F^* + B G) x_2 &= \lambda x_2, 
\end{align*}
which can be equivalently written in the following form:
\begin{align}\label{new}
\big(F-\lambda I\big) \big(x_1+Gx_2\big) + \lambda  G x_2&= 0, \\
B \big(x_1+Gx_2\big) + (F^* -\lambda I) x_2 &= 0, \label{newnew}
\end{align}
 i.e. 
 \begin{align}
\begin{bmatrix}
B & F^*-\lambda I\\[2pt]
F-\lambda I & \lambda G
\end{bmatrix}
\begin{bmatrix} x_1+Gx_2 \\ x_2  \end{bmatrix} &= 0, \label{eq:5}
\end{align}
which is precisely  the Hermitian generalized eigenvalue problem \eqref{eq:gep}.

Furthermore, assuming $U^*U=I$,    equations  \eqref{new}-\eqref{newnew} can be rewritten as
\begin{align}\label{eq:qep_temp1}
F(x_1+  x_2)&= \lambda  x_1\\
\bigl[
\lambda^2 I
-
\lambda\bigl(F+ F^*+B\bigr)
+
F^*F
\bigr]\big(x_1+x_2\big)
&=0,
\label{eq:qep_temp2}
\end{align}
Using definitions of  $B$ and $F$, after some algebra, equation 
 \eqref{eq:qep_temp2} becomes
\[
\Big(\eta^{2} \,I + \eta\,(z U^* V + \bar{z} V^* U\!-\!V^*V) + |z|^{2}V^* (UU^*\!-\!I)V\Big)(x_1+x_2)=0,
\]
implying that $\eta =\lambda -|z|^2$ solves QEP \eqref{eq:qep}. 

To prove the converse, recall that  $\mu_{U,V}(z) \ge 0$ holds for every $z$, while  $\mu_{U,V}(z) = 0$ if and only if $z\in \Lambda(UV^*)=\Lambda(V^*U) \cup \{0\}$. Therefore, since we assumed $z\notin \Lambda(V^*U) $, either  $z=0$, in which case $C=0$, too, or $\lambda\geq\mu_{U,V}(z)>0$. However, in both cases equation \eqref{eq:qep_temp1}  holds true for appropriately chosen $x_1$, implying that $\lambda$ solves GEP \eqref{eq:gep}.
$\Box$ 
This result has significant practical importance. While computing $\mu_{U,V}$ from EVP \eqref{eq:lrm} is most straightforward,  potential caveat lies in the fact that $\lrm_{U,V}$ might inherit nonormality of $A=UV^*$, yielding unreliable computation of $\mu_{U,V}$. While the pencil in \eqref{eq:gep} is Hermitian, unfortunately its GEP is not a definite one, and, hence, the eigenvalues  can be badly conditioned.  Contrary to that,  computing $\mu_{U,V}$ from QEP \eqref{eq:qep} is always reliable, since its hyperbolic structure ensures real solutions and the orthogonality of eigenvector basis. Note that  we can always have  $U^*U=I$, using the economy QR decomposition $A=QR$ computed via modified Gram-Schmidt algorithm with column pivoting. While computational complexity of each one will differ, all of them will are of order $\Oc(r^2d)$, significantly improving $\Oc(d^3)$ for computing $\sigma_{\min}(A-z I)$. 

Importantly, this extends to other key quantities of interest, such as distance to instability given in \eqref{eq:d2i}, equivalently expressed as 
$$\dti(A):=\min_{\varphi\in[0,2\pi]} \norm{(e^{\ii \varphi} I - A)^{-1}}^{-1}.    
$$
Indeed, direct approach to compute $\dti(A)$ via Theorem \ref{thm:main} is to solve nonconvex optimization problem $\min_{\varphi\in[0,2\pi]} \lambda_{min} \big(\lrm_{U,V}(\rho\,e^{\ii\varphi})\big)$ by setting $\rho=1$ via gradient descent algorithms, which can be done recalling that if $x,y\in\C^{2r}$ are right and left eigenvector of $\lrm_{U,V}(\rho\,e^{\ii\varphi})$ corresponding to a simple eigenvalue, then the derivative can be computed via {\cite{Li2014_HLA_Perturbation}}
\begin{equation}\label{eq:dlambda}
    \partial_{\varphi} \mu_{U,V}(\rho\,e^{\ii\varphi})\big) =\rho\, \frac{e^{\ii\varphi -\pi/2}(U y_1)^* V(x_1+x_2)+e^{-\ii\varphi+\pi/2}(V y_2)^* U x_2 }{2\,\mu_{U,V}(\rho\,e^{\ii\varphi})\,[y_1^* x_1 + y_2^* x_2]}. 
\end{equation}

Similarly, we can also compute the relevant derivative for the time-continuous case:
\begin{equation}\label{eq:dlambda_cont}
    \partial_{\omega}\mu_{U,V}(a{+}\ii\omega) {=}\frac{2\omega\,(y_1^*x_1+y_1^*x_2+y_2^*x_2){-}\ii(U y_1)^* V(x_1{+}x_2){+}\ii(V y_2)^* U x_2 }{2\,\mu_{U,V}(a{+}\ii\omega)\,[y_1^* x_1 {+} y_2^* x_2]}. 
\end{equation}

However, since gradient methods converge only locally, additional algorithmic tools are needed to ensure global optimization. An important class of globally convergent algorithms for pseudospectral optimization is based on the \emph{criss-cross} (or branch-and-bound) approach introduced for computing the distance to instability by Byers \cite{byers1988bisection} and later extended to compute the Kreiss constant by Burke,  Lewis and Overton \cite{BLO2003crisscross}. These methods exploit the property that the $\varepsilon$-pseudospectrum $\Lambda_{\varepsilon}(A)$ is a closed, semianalytic set whose boundary can be characterized via eigenvalue problems. The crucial computational step in these algorithms is the ability, given $\varepsilon > 0$, to find the discrete set of intersection points of $\Lambda_{\varepsilon}(A)$ with either the unit circle (for discrete-time systems) or the imaginary axis (for continuous-time systems). These intersection points correspond to eigenvalues of certain structured matrix pencils whose dimension equals that of $A$. To enable efficient low-rank versions of these criss-cross algorithms, we develop in the following two key characterizations that reduce these large-scale eigenvalue computations to problems involving only matrices of the rank of the perturbation.

\begin{proposition}\label{prop:d2i}
Given $U,V\in\C^{d,r}$, $r \leq d$, let $A =UV^* \in \mathbb{C}^{d,d}$, then for all $\rho>0$, $\eps \geq 0$ and $\varphi\in[0,2\pi)$, $\rho\,e^{\ii \varphi} \in \partial \PSpec(A)$, if and only if  $e^{-\ii \varphi}$ solves the generalized eigenvalue problem
\begin{equation}\label{eq:gep_d2i}
\left[
\begin{array}{cc}
\!\!\rho\,U^*V  & \rho\,U^*U\!\!\\ 
\!\!V^*V & (\rho^2{-}\eps^2)I {+} V^*V\,U^*U\!\!
\end{array}
\right]\!\!\!
\left[
\begin{array}{c}
\!\!x_1\!\!\\ 
\!\!x_2\!\!
\end{array}
\right]\! {=} e^{-\ii \varphi}\left[
\begin{array}{cc}
\!\!(\rho^2{-}\eps^2) I  & \rho^2\,U^*U\!\!\\ 
0 & \rho\,V^*U\!\!
\end{array}
\right]\!\!\!
\left[
\begin{array}{c}
\!\!x_1\!\!\\ 
\!\!x_2\!\!
\end{array}
\right]\!\!,
\end{equation}
where $x=\left[
\begin{array}{c|c}
x_1^* & 
x_2^*
\end{array}
\right]^* \in \mathbb{R}^{2r}$ is the eigenvector of $\lrm_{U,V}(\rho\,e^{\ii\varphi})$ associated to $[\mu_{U,V}(\rho\,e^{\ii\varphi})]^2$.
\end{proposition}
{\em Proof:} 
Observe that according to Theorem \ref{thm:main} $\rho\,e^{\ii \varphi} \in \partial \PSpec(UV^*)$ is equivalent to 
$\lambda_{min}(\lrm_{U,V}(\rho e^{\ii \varphi}))=  \eps^2$. That is, to
$$
\left[
\begin{array}{cc}
(\rho^2-\eps^2)I -\rho\,e^{\ii \varphi}U^*V* &\hspace{-0,5cm} (\rho^2 I- \rho\,e^{\ii \varphi}U^*V)\,U^*U\vspace*{0.3cm}\\ 
V^*V &\hspace{-0,5cm} (\rho^2-\eps^2)I -\rho\,e^{-\ii \varphi}V^*U + V^*V\,U^*U 
  \vspace*{0.1cm}
\end{array}
\right] \left[
\begin{array}{c}
x_1\\ 
x_2
\end{array}
\right] =  0,
$$ 
i.e.
{\small $$
\left[
\begin{array}{cc}
I & 0 \vspace*{0.3cm}\\ 
0 & -e^{\ii \varphi}I
  \vspace*{0.1cm}
\end{array}
\right]\hspace*{-0,2cm}
\left[
\begin{array}{cc}
\hspace{-0,1cm} (\rho^2-\eps^2)I -\rho\,e^{\ii \varphi}U^*V & \hspace{-0,9cm} (\rho^2\,I- \rho\,e^{\ii \varphi}U^*V)\,U^*U\hspace{-0,1cm}  \vspace*{0.3cm}\\ 
\hspace{-0,1cm}  V^*V & \hspace{-0,9cm} (\rho^2-\eps^2)I -\rho\,e^{-\ii \varphi}V^*U + V^*V\,U^*U \hspace{-0,1cm} 
  \vspace*{0.1cm}
\end{array}
\right]\hspace*{-0,15cm} \left[
\begin{array}{c}
x_1\\ 
x_2
\end{array}
\right]\hspace*{-0,1cm} = \hspace*{-0,1cm} 0,
$$
and, hence to
$$
\Bigg( 
\left[
\begin{array}{cc}
(\rho^2-\eps^2)I  & \rho^2\,U^*U \vspace*{0.3cm}\\ 
0 & \rho\,V^*U 
  \vspace*{0.1cm}
\end{array}
\right] -
e^{\ii \varphi}
\left[
\begin{array}{cc}
\rho U^*V & \rho\,U^*V\,U^*U \vspace*{0.3cm}\\ 
V^*V & (\rho^2-\eps^2)I  + V^*V\,U^*U 
  \vspace*{0.1cm}
\end{array}
\right] \Bigg)\left[
\begin{array}{c}
x_1\\ 
x_2
\end{array}
\right]= 0 ,
$$
i.e.
$$
\left[
\begin{array}{cc}
\rho\,U^*V & \rho\,U^*V \,U^*U\vspace*{0.3cm}\\ 
V^*V & (\rho^2-\eps^2)I  + V^*V\,U^*U 
  \vspace*{0.1cm}
\end{array}
\right] \left[
\begin{array}{c}
x_1\\ 
x_2
\end{array}
\right] =e^{-\ii \varphi}
\left[
\begin{array}{cc}
(\rho^2-\eps^2)I  & \rho^2\,U^*U\vspace*{0.3cm}\\ 
0 & \rho\,V^*U 
  \vspace*{0.1cm}
\end{array}
\right] 
\left[
\begin{array}{c}
x_1\\ 
x_2
\end{array}
\right].
$$}
Thus the proof is completed.
$\Box$

Remark that the GEP problem \eqref{eq:gep_d2i} is regular whenever $V^*U$ is of full rank. Moreover, due to it's structure efficient solvers with complexity of $\Oc(r^3)$ can be applied to obtain possible candidate points of intersection of $\eps$-pseudospectrum with the unit circle, improves the complexity $\Oc(d^3)$ of algorithms such as the one in {\cite{BLO2003crisscross}} to the overall all complexity of this approach $\Oc(rd^2)$.

\begin{proposition}\label{prop:d2i_cont}
Given $U,V\in\C^{d,r}$, $r \leq d$, let $A =UV^* \in \mathbb{C}^{d,d}$, then for all $a, b\in\R$ and $\eps \geq 0$, $a+\ii b \in \partial \PSpec(A)$, if and only if $b$ solves the quadratic eigenvalue problem $(b^2 I - B_1 b - B_0) w = 0$, where
\begin{equation}\label{eq:gep_d2ii}
B_1 {=} \left[
\begin{array}{cc}
\ii\,U^*V  & \hspace{-0.3cm} 0\\ 
0 & \hspace{-0.3cm}  -\ii\,V^*U
\end{array}
\right]\!,\;
B_0{=}\!\left[
\begin{array}{cc}
(\varepsilon^2\!-\!a^2)I \!+\!a\,U^*V &\hspace{-0.3cm} -\varepsilon^2\,U^*U\\ 
-V^*V &\hspace{-0.3cm}  (\varepsilon^2\!-\!a^2)I \!+\!a\,V^*U
\end{array}
\right]
\end{equation}
and  $w=\left[
\begin{array}{c|c}
x_2^*& 
x_1^*+(U^*U\,x_2)^* 
\end{array}
\right]^*$, with $x=\left[
\begin{array}{c|c}
x_1^* & 
x_2^*
\end{array}
\right]^* \in \mathbb{R}^{2r}$ being the eigenvector of $\lrm_{U,V}(a+\ii b)$ associated to $[\mu_{U,V}(a+\ii b)]^2$.
\end{proposition}
{\em Proof:} 
The proof follows directly from GEP \eqref{eq:gep} by permuting the block rows.  Indeed, $a+\ii b \in \partial \PSpec(A)$ means that $\mu_{U,V}(a+\ii b)=  \eps$  and \eqref{eq:gep} can be rewritten as
\[
\left(\begin{bmatrix}
\eps^2I  & - \eps^2U^*U  \\[2pt]
 0 &\eps^2 I
\end{bmatrix} \hspace{-0.1cm}
 - \hspace{-0.1cm}
\begin{bmatrix}
(a^2\!+\!b^2)\!I-\! a\,U^*V \!-\! \ii b\,U^*V& \hspace{-1.5cm} 0    \\[2pt]
V^*V & \hspace{-1.5cm}  (a^2\!+\!b^2)I\!-\! a\,V^*U \!+\! \ii b\,V^*U
\end{bmatrix}\right)\hspace{-0.1cm}
w =  0,
\]
i.e.
\[
\Bigg(-b^2I + b\begin{bmatrix}
\ii \,U^*V& \hspace{-0.3cm} 0    \\[2pt]
0& \hspace{-0.3cm} - \ii \,V^*U
\end{bmatrix}+\begin{bmatrix}
(\eps^2-a^2)I+a\,U^*V & \hspace{-0.5cm} - \eps^2U^*U    \\[2pt]
-V^*V & \hspace{-0.5cm}  (\eps^2-a^2)I+a\,V^*U 
\end{bmatrix}
\Bigg) w =  0,
\]
which completes the proof.
$\Box$ 

To conclude, we note that, concerning the transient growth of non-normal dynamical systems, Theorem \ref{thm:main} coupled with Propositions \ref{prop:d2i} and \ref{prop:d2i_cont}, for $\rho>1$ and $a>0$, allows one to speed up the Kreiss constant computational algorithms, such as \cite{mitchell2020computing, apkarian2020optimizing}.

\section{Low rank approximation of pseudospectral properties}\label{sec:estimation}

Until now, we have managed to scale the complexity of pseudospectral computation from standard $\Oc(d^3)$ to $\Oc(r^2 d)$ per point in $\C$, with $r$ being the smaller dimension in the given low rank form. If such form is not available, one obtains additional $\Oc(r d^2)$ complexity from pre-computing low rank form such as QR decomposition. However, this might still be prohibitive  for large scale systems with large rank. To mitigate this, we analyze how inexact low rank forms, or in another terms low-rank approximations, affect the computation of the pseudospectrum. 

The following lemma is the well-known fact about perturbation of the pseudospectrum. 

\begin{lemma}[\cite{TE2020}, Theorem 52.4]\label{lm:ps_perturbation} Let $A,B\colon\spH\to\spH$ be two bounded linear operators on a separable Hilbert space $\spH$. Then ,  $\Spec{\eps}(A)\subseteq \Spec{\eps+\norm{A-B}}(B)$. Moreover, $\Spec{\eps-\norm{A-B}}(B)\subseteq\Spec{\eps}(A)\subseteq \Spec{\eps+\norm{A-B}}(B)$ holds true whenever $\varepsilon>\norm{A-B}$.
\end{lemma}
Based on this Lemma, from Theorem \ref{thm:main} we immediately get
\begin{lemma}\label{lemma:svdlast} For arbitrary matrices $A{\in} \mathbb{R}^{d,d}$, and $U,V{\in}\R^{d,\ell}$, $\ell{\leq} d$, if $\norm{A-UV^*}\le\eps $, it holds that $
\underline{\Theta}_\eps(A,U,V) \subseteq\PSpec(A) \subseteq \overline{\Theta}_\eps(A,U,V),
$ where
\[
 \Theta^{\pm}_\eps(A,U,V):=\big\{z \in \C \; :
\mu_{U,V}(z)\leq \eps \pm \norm{A-UV^*}\big\}.
\]
\end{lemma}
As one can anticipate, we will combine this result with the best low rank approximations singular value decomposition (SVD) to derive arbitrary good localizations of the pseudospectrum of any matrix via its low-rank approximation. To that end, in the following for an arbitrary {$\Amxx$},  the SVD is a factorization $A=U_\ell  \Sigma_\ell  V_\ell^*$, where $\ell=\min\{m,d\}$,  $U_\ell\in\C^{m,\ell}$ and $V_\ell\in\C^{d,\ell}$ are matrices with orthonormal columns and $\Sigma_\ell \in {\mathbb{R}_+^{\ell,\ell}}$ is a diagonal matrix with all diagonal entries  nonnegative and arranged in  descending order.  Diagonal entries of  matrix $\Sigma$ are  singular values,   columns of  $U_\ell$  are  left singular vectors and columns of $V_\ell$ are right singular vectors of matrix $A$. Furthermore, for $\ell<\min\{m,d\}$ matrix $U_\ell  \Sigma_\ell  V^*_\ell$ is known as a $\ell$-truncated SVD of $A$, and, due to Eckart-Young theorem, the approximation error for this truncation is $\norm{A-U_\ell  \Sigma_\ell  V^*_\ell}=\sigma_{\ell+1}(A)$. So, to summarize, a direct consequence of Theorem \ref{thm:main} with the above observations on SVD and Lemma \ref{lm:ps_perturbation} yields.

\begin{theorem}\label{thm:svdlast} For arbitrary matrices $A{\in} \mathbb{R}^{d,d}$, and $\ell{\leq} d$, let $U_\ell\Sigma_\ell V_\ell^*$,  be the  truncated SVD of $A$. Then 
\begin{equation} \label{eq:ps_loc}
{\Theta}^{-}_{\ell,\eps}(A) \subseteq\PSpec(A) \subseteq {\Theta}_{\ell,\eps}^{+}(A),
\end{equation}
where
$
 {\Theta}^{\pm}_{\ell,\eps}(A)={\Theta}^{\pm}_\eps(A, U_\ell \Sigma_\ell, V_\ell) = \big\{z \in \C \; :
\mu_{U_\ell,V_\ell\Sigma_\ell}(z)\leq \eps \pm \sigma_{\ell+1}\big\}$.

Consequently,  $ \PSpec(A)$ is localized by the set of points $z\in\C$ such that the smallest eigenvalue of the quadratic Hermitian eigenvalue problem  
\begin{equation}\label{eq:qep_svd}
\lambda^{2} \,I \;+\; \lambda\,(z U_\ell^* V_\ell \Sigma_\ell + \bar{z}\Sigma_\ell V_\ell^* U_\ell-\Sigma_\ell^2) \;+\; |z|^{2} \,\Sigma_\ell( V_\ell^* U_\ell U_\ell^* V_\ell -I)\Sigma_\ell.
\end{equation}
is bounded bellow by $|z|^2 {+} (\eps-\sigma_{\ell+1})^2$ and above by $|z|^2 {+} (\eps+\sigma_{\ell+1})^2$. Moreover, $\PSpec(A) = {\Theta}^-_{\ell,\eps}(A)= \overline{\Theta}^+_{\ell,\eps}(A)$ whenever $A$ has rank $\ell$. 
\end{theorem}
For the sake of completeness, we remark that using the localization sets ${\Theta}^-_{\ell,\eps}(A)$ and ${\Theta}^+_{\ell,\eps}(A)$ instead of the pseudospectrum provides a lower and upper bound, respectively, for the distance to instability, so, the following known perturbation result of the Kreiss constant, enables the application of our results not only to low rank characterizations, but also to low-rank approximations of linear and non-linear dynamical systems, presented in the following sections, respectively.

\begin{lemma}[\cite{Kostic-ICML2024}] \label{lm:kreiss_perturbation} Let $A,B\colon\spH\to\spH$ be two bounded linear operators on a separable Hilbert space $\spH$. If $\norm{A-B} < \dti(A)$, then  $${\abs{\kreiss(A)-\kreiss(B)}} / {\kreiss(A)} \leq {\norm{A-B}}/{(\dti(A) - \norm{A-B})}.$$
\end{lemma}
We remark that while the proof in \cite{Kostic-ICML2024} is derived for discrete-time Kreiss constant related to matrix powers, the same follows for the continuous time Kreiss constant. 

In the reminder of this section we 
show how to obtain good estimation of the pseudospectrum with high probability via randomized SVD algorithms. The approach is summarized in the following probabilistic corollary of Theorem \ref{thm:svdlast}.

\begin{corollary}\label{lemma:randomized}
Given $A\in\R^{d,d}$, let $\ell<k<d$ and $Q\in\R^{d,k}$ be a random matrix such that $Q^*Q=I$. If $\PP\{\norm{(I-Q^* Q)A} > \alpha_k(A)\}\leq \delta$ holds for some $\delta\in(0,1]$, and $Q^* A = \hat{U}_k\hat{\Sigma}_k \hat{V}_k^*$ is the SVD of $Q^*A$, then with probability at least $1-\delta$ it holds that
 \begin{equation}\label{eq:svdrand}
\PSpec(A) \!\subseteq\! \Big\{z \in \C \colon
\mu_{Q\hat{U}_\ell,\hat{V}_\ell\hat{\Sigma}_\ell}(z) \leq  \eps + {\hat{\sigma}_{\ell+1}}+ \alpha_k(A) \Big\} .
\end{equation}
\end{corollary}

In another words, tail bounds for the operator norm error of random projections are easily transformed into efficient pseudospectral estimation. Typical procedure for constructing random matrix $Q$ is via \textit{randomized rangefinder} method, for which error bounds have been extensively studied, see e.g. \cite{Martinsson2020} and references therein. 

For example, in the case of vanilla range finder based on standard Gaussian sketching, that is $\Omega\sim\mathcal{N}(0,I)$ and $QR = A\Omega$ is QR decomposition, from \cite[Theorem 11.5]{Martinsson2020} we have that
\[
\EE \norm{(I-Q^* Q)A} \leq \big(1+\sqrt{\tfrac{\ell}{k-\ell-1}}\big)\sigma_{\ell+1} + \tfrac{e\sqrt{k}}{k-\ell} \sqrt{\textstyle{\sum_{j=\ell+1}^n
} \sigma_j^2},
\]
 for $\ell<k-1$, which by applying Markov's inequality $
\PP \{\norm{(I-Q^* Q)A} \geq \varepsilon'\}\leq \tfrac{\EE \{\norm{(I-Q^* Q)A} }{\varepsilon'}$,  and solving in probability gives 
$$
\alpha_k(A) = \frac{1}{\delta}\Big(\big(1+\sqrt{\tfrac{\ell}{k-\ell-1}}\big)\sigma_{\ell+1} + \tfrac{e\sqrt{k}}{k-\ell} \sqrt{\textstyle{\sum_{j=\ell+1}^n
} \sigma_j^2}\Big).$$

Another interesting case is of the sketching with subsampled randomized trigonometric transforms, see \cite[Theorem 11.6]{Martinsson2020}, where one has
\[
\alpha_k(A) = (1+3\sqrt{d / k})\sigma_{\ell+1},\; \text{ for } {\delta} = \Oc(1/\ell) \;\text{ and }\; \ell \geq 8 (\ell+8\log(\ell d))\log \ell.
\]
Therefore, setting $\varepsilon'_\ell=\frac{3+6\sqrt{d / \ell}}{2+3\sqrt{d / k}}\,\hat{\sigma}_{\ell+1}$ we can compute the estimation 
$\PSpec(A) \subseteq \big\{z \in \C \colon
\mu_{Q\hat{U}_\ell,\hat{V}_\ell\hat{\Sigma}_\ell}(z) \leq  \eps + \eps_\ell' \,\big\}$
that holds with failure probability $\Oc(1/\ell)$.

\smallskip

We conclude this section with a numerical demonstration of how useful are the results presented so far. 

\begin{example}\label{ex:toy_1}
We consider matrices $A=UV^\top\in\R^{d,d}$ of rank $r$, constructed in two following ways. First, for the left plot in Figure \ref{fig:toy_1}, we set $d=100$ and $r=5$ and draw entrees of matrix $\Omega\in\R^{d,r}$ and vector $\omega\in\R^d$ from the standard Gaussian distribution. Then we construct $U$ via QR decomposition $UR=\Omega$ and $V$ as $U\Diag(e^{2\pi\ii\omega_1},e^{2\pi\ii\omega_2},e^{2\pi\ii\omega_3},e^{2\pi\ii\omega_4}, e^{2\pi\ii\omega_5/2},\ldots,e^{2\pi\ii\omega_{r}/(r-4)})$. On the other hand, for the right plot in Figure \ref{fig:toy_1}, we simply construct  $U$ and $V$ by sampling their columns uniformly on $d$-dimensional sphere. 
\end{example}

\begin{figure}[ht!]
    \centering
   \includegraphics[scale=0.475]{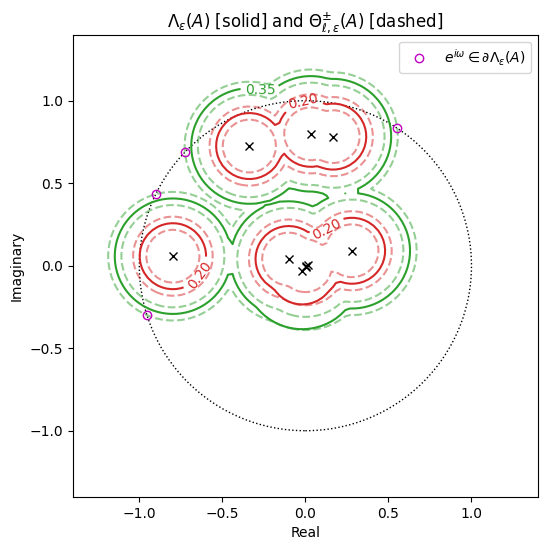}\;\;\;\;
      \includegraphics[scale=0.475]{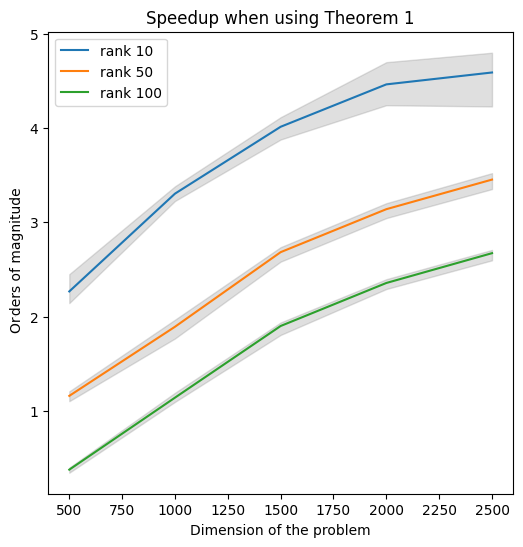}
    \caption{Results for the Example \ref{ex:toy_1}. Left panel: pseudospectrum of a $100\times 100$ rank 10 matrix with its localization set based on rank $\ell=6$ truncated SVD. The test for the intersection (magenta markers) is performed by applying Proposition \ref{prop:d2i}. Right pannel: Comparison of the computational time speedup (logarithm of the CPU time ratio between standard computation and the low-rank version) for various dimensions and ranks over 10 random trials.}
    \label{fig:toy_1}
\end{figure}

\section{Nonlinear stochastic dynamics}\label{sec:koopman}

In this section we demonstrate how transient dynamics in stable nonlinear dynamical systems can be inferred by estimating the  pseudospectrum of transfer operators. That is, we collect samples from a trajectory ${\cal D}_n=(x_i)_{1\leq i\leq n}$ of a discrete (stochastic) non-linear system $(X_t)_{t\in\N}\subseteq \R^d$, and learn it's global linearization  $\dt{t}\!-\!\cf = \adjcKoop(\dt{t-1}\!-\!\cf) = (\adjcKoop)^t(\dt{0}-\cf)$, where $\cKoop\colon\Lii\to\Lii$, $\pi$ being the invariant measure, is given by $f\mapsto \EE[f(X_{t+1}\,\vert\,X_{t}=\cdot] + \EE_{X\sim\pi}[f(X)]$, and 
 $\dt{t}:=d{\pi_{t}} / d\pi \in \Lii$, $\pi_t$ being the law of $X_t$. The asymptotically stable dynamics governed by $\cKoop$ is learned in a form of an operator $\EEstim\colon\Hc\to\Hc$ defined on a Hilbert space of functions $\Hc\subseteq \Lii$ that minimizes specific risk functional in order to control the error $\norm{A_{\pi_{\vert_{\Hc}}} - \EEstim}_{\Hc\to\Lii}$ measured in the operator norm from $\Hc\to\Lii$, \cite{Kostic2022}. Namely, the empirical risk minimization for the operator regression problem \cite{Kostic2022} is formulated as
\begin{equation}\label{eq:risk}
\min_{\EEstim\colon\Hc\to\Hc} \widehat{\mathcal{R}}_\reg(\EEstim)=\tfrac{1}{n}\sum_{i\in[n-1]}\norm{\hat{k}_{x_{i+1}} - \EEstim^*\hat{k}_{x_i}}^2_{\Hc} + \reg\norm{\EEstim}^2_{\text{HS}}, 
\end{equation}
where $\reg>0$ is the Ridge regularization parameter, $k_x = k(x,\cdot)\in\Hc$ is the embedding of the point $x\in\R^d$ into RKHS $\Hc$, and $\hat{k}_{x_i} = k_{x_i} - \bar{k}_x$ are empirically centered features, with $\bar{k}_x=\tfrac{1}{n}\sum_{j\in[n]}k_{x_j}$ being the empirical mean.

 The arbitrarily good approximation of $\cKoop$ is guaranteed whenever $\Hc$ is (an infinite-dimensional) reproducing kernel Hilbert space (RKHS) defined by a universal symmetric and positive definite kernel function  
$k:\R^d\times\R^d \to \R$~\cite{aron1950, Steinwart2008}. On the other hand, according to \cite{Kostic-ICML2024} the forecasting error bounds are highly impacted by $\min\{\pcon(\cKoop)\sumcon(\EEstim), \pcon(\EEstim)\sumcon(\cKoop)\}$, where $\pcon$ and $\sumcon$ are given in \eqref{eq:constants}. The objective of this section is to demonstrate how pseudospctra of $\cKoop$ and $\EEstim$ can be estimated, so that the above quantities can be determined. 

We first analyze operator $\EEstim\colon\Hc\to\Hc$. As shown in series of works on the subject, difference between geometries of two Hilbert spaces, hypothetical (known) domain $\Hc$ and the true (unknonwn) domain $\Lii$, presents a challenge in spectral estimation of $\cKoop$, which can be mitigated by controlling the rank of $\EEstim$, \cite{Kostic2022}. This led to the study of a vector-valued Reduced Rank Regression (RRR) algorithm which solves problem \eqref{eq:risk} with a hard rank constraint on $\EEstim$, for which strong statistical guarantees have been proven, \cite{Kostic2023,Kostic-ICML2024}.

Introducing  the sampling operators for data $\cal{D}_n=(x_i)_{i\in[n]}$ and RKHS $\Hc$ and their adjoints by
\[
    \ES \colon \Hc \to \R^{n}  \text{ s.t. }  f \mapsto \tfrac{1}{\sqrt{n}}[ f(x_{i})]_{i \in[n]}  \; \text{ and} \; \ES^* \colon \R^{n} \to \Hc  \text{ s.t. } w \mapsto \tfrac{1}{\sqrt{n}}\sum_{i\in[n]}w_i k_{x_i},
\]
it has been shown, see \cite{Kostic-ICML2024,Kostic2023_Spectral}, that RRR estimator with hyperparameters $r\leq n$ and $\reg>0$ is given by 
\begin{equation}\label{eq:rrr}
\EEstim_{\gamma,r} = \ECreg^{-1/2}[\![\ECreg^{-1/2} \ECxy]\!]_r, \; \text{where } [\![\cdot]\!]_r \text{ denotes the } r\text{-truncated SVD},
\end{equation}
$\ECx = \ES^*J_n E_n^\top E_n J_n\ES$ denotes the empirical covariance and $\ECxy = \ES^*J_nE_n^\top E_n^\top J_n\ES$ cross-covariance operator, for $J_n=(I-\one_n\one_n^\top)$ being the projector orthogonal to $\one_n = n^{-1/2} [1,\ldots,1]^\top\in\R^n$, $E_n= \tfrac{\sqrt{n}}{\sqrt{n-1}}[e_2\,\vert\,\cdots\,\vert\,e_n\,\vert\,0]\in\R^{n,n}$ being the scaled left shift matrix and $\ECreg = \ECx +\reg I$.

The following result shows how to compute the pseudospectrum of $\EEstim_{\reg,r}$ by the means of Proposition \ref{thm:main}. To present it, let us further denote kernel Gram matrix $\Kx\!:=\! \ES\ES^* \!= \!\tfrac{1}{n}[k(x_i,\!x_j)]_{i,j\in[n]}$ and its centered version $\CKx\!:=\! J_n\ES\ES^*J_n$. 

Before, we prove it, note that due to different geometry between true domain $\Lii$ and effective learning domain $\spH$, one typically expects different stability robustness and transient behavior. Indeed, while we have unique estimation of $\Koop$'s eigenvalues by the ones of $\EEstim$, psuedospectral properties are norm dependent, and, hence, differently estimated. 

\begin{theorem}\label{thm:ps_rrr}
Let $\EEstim_{\reg,r}$ be Reduced Rank Regression estimator of $\cKoop$ given by \eqref{eq:rrr}. If $V_r = \CKx U_r$ for the columns of $U_r = [u_1\vert\ldots\vert\,u_r]\in\R^{n\times r}$  and $\Sigma_r^2 =\Diag(\sigma_1^2,\ldots,\sigma_r^2)$ solving the generalized eigenvalue problem
\begin{equation}\label{eq:rrr_compute}
     E_n^\top\CKx E_n  \CKx u_i {=} \sigma_i^2 (\CKx{+}\reg I) u_i,\; \text{ normalized s.t. } \; u_i^\top \CKx(\CKx{+}\reg I)u_i {=} 1,\,i\in[r],
\end{equation}
where $\sigma_i^2>0$, $i\in[r]$, are the largest eigenvalues, 
then the pseudospectrum of the estimator in the space $\spH$ is characterized as  $\PSpec(\EEstim_{\reg,r}) = \left\{z\in\C\,\vert\, \widehat{\mu}_{\spH}(z)\leq \varepsilon \right\}$, where $[\widehat{\mu}_{\spH}(z)]^2$ is the smallest eigenvalue of 
\begin{equation}\label{eq:rkhs_lrps}
\left[
\begin{array}{cc}
|z|^2I -z \, (E_n V_r)^\top V_r & |z|^2\,U_r^\top V_r -z \, (E_n V_r)^\top V_r\, U_r^\top V_r \vspace*{0.3cm}\\ 
(U_r+\reg V_r)^\top U_r \Sigma_r^2 & |z|^2I -\overline{z} V_r^\top E_n V_r + (U_r+\reg V_r)^\top U_r \Sigma_r^2\, U_r^\top V_r
  \vspace*{0.1cm}
\end{array}
\right],
\end{equation}
while the empirical estimation of the $\Lii$ pseudospectrum of $\cKoop$ is given by $$\left\{z\in\C\,\vert\, \widehat{\mu}_{\Lii}(z)\leq \varepsilon \right\},$$ where $[\widehat{\mu}_{\Lii}(z)]^2$ is the smallest eigenvalue of 
\begin{equation}\label{eq:l2_lrps}
\left[
\begin{array}{cc}
|z|^2I -z \, (E_n V_r)^\top V_r & |z|^2\,U_r^\top V_r -z \, (E_n V_r)^\top V_r\, U_r^\top V_r \vspace*{0.3cm}\\ 
(E_n V_r)^\top (E_n V_r) & |z|^2I -\overline{z} \, V_r^\top E_n V_r  + (E_n V_r)^\top (E_n V_r) U_r^\top V_r
\vspace*{0.1cm}
\end{array}
\right].
\end{equation}
\end{theorem}
{\em Proof:} 
Start by observing that according to \cite{turri2026randomized} and \cite[Theorem A.2]{Kostic-ICML2024} we have that the learned RKHS operator  is $\EEstim_{\reg,r} = \ES^* J_n U_r V_r^\top E_n^\top J_n\ES$. Hence, applying Theorem \ref{thm:main} and using $V_r^\top E_n^\top \CKx E_n V_r = (U_r+\reg V_r)^\top U_r \Sigma_r^2$ due to \eqref{eq:rrr_compute}, completes the proof of \eqref{eq:rkhs_lrps}. On the other hand, to prove \eqref{eq:l2_lrps}, start by noting that for the $\Lii$-operator norm of an operator $\EEstim\colon\spH\to\spH$ we have that $\norm{\EEstim}_{\Lii\to\Lii} = \norm{(\TS^* \TS)^{1/2}\EEstim (\TS^* \TS)^{\dagger/2}}_{\spH\to\spH}$, \cite{Kostic2023}. Thus, we can estimate the $\Lii$-pseudospectrum by applying {Theorem} \ref{thm:main} to $$(\ES^*J_n\ES)^{1/2}\ES^*J_n U_rV_r^\top E_n^\top J_n\ES (\ES^*J_n\ES)^{\dagger/2} = \ES^*J_n \CKx^{1/2}U_rV_r^\top E_n^\top\CKx^{\dagger/2}J_n\ES.$$ But, $\CKx\CKx^{\dagger}$ is the orthogonal projector onto $\range(\CKx)\supseteq \range(E_n V_r)$, implying that  $V_r^\top E_n^\top \CKx\CKx^{\dagger} E_n V_r = V_r^\top E_n^\top E_n V_r$. 
Therefore, since $$(\ES^*J_n\ES)^{1/2}(\ES^*J_n\ES)^{\dagger/2}\ES^*J_n = \ES^*J_n$$ implies $U_r^\top J_n\ES(\ES^*J_n\ES)^{1/2}(\ES^*J_n\ES)^{\dagger/2}\ES^*J_n U_r = U_r^\top V_r$ \eqref{eq:l2_lrps} follows.
$\Box$ 

To conclude, note that direct computation of 
$\sigma_{\min}^2(\EEstim_{\reg,r})$ 
would scale as $\bigO(n^3)$ per point $z$, which is unfeasible with typical samples sizes of order $10^4$. On the other hand, once a good low rank estimator $\EEstim$ is learned in the form of low rank factors $U_r$ and $V_r$, which can be efficiently done using \href{https://kooplearn.readthedocs.io/latest/}{Kooplearn} software \cite{turri2025kooplearn}, our approach  allows even a brute force computation of pseudospectra over a fine grid. Indeed, noting that multiplication with $E_n$ encodes just the shift of rows, the total cost of computing the pseudospectrum of $\EEstim_{\reg,r}$ for $m$ grid points using \eqref{eq:rkhs_lrps} or \eqref{eq:l2_lrps} is $\mathcal{O}(r^2\,n+m\,r^3)$, allowing one to take $m\asymp n/r$ without incurring extra costs w.r.t. computing the spectrum. This is implemented in the following numerical examples.

\begin{figure}[ht!]
    \centering
   \includegraphics[scale=0.44]{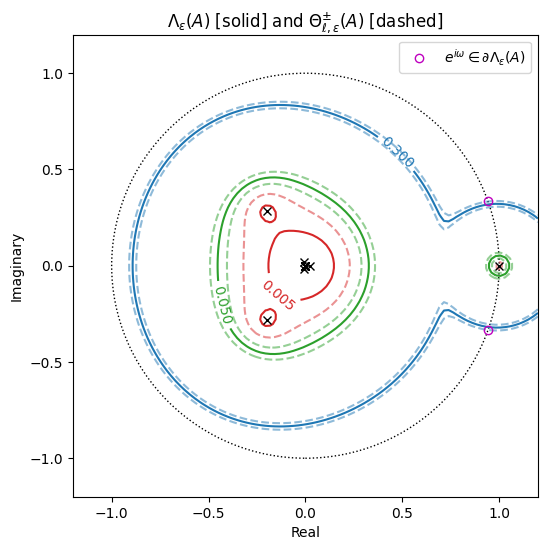}
      \includegraphics[scale=0.44]{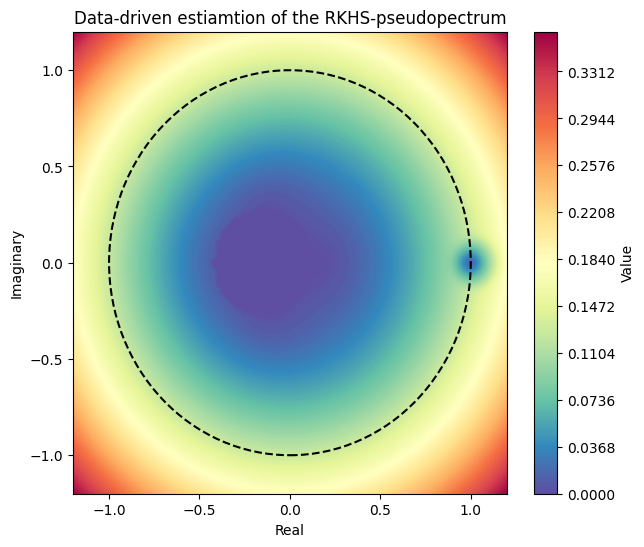}
    \caption{Pseudospectrum of a rank 21 integral operator of Example \ref{ex:logistic} and its localization set based on rank $r=7$ truncated SVD. On the left is the ground truth computed numerically, while on the right is the result obtained by learning from data. The test for the intersection (magenta markers) is performed by applying Proposition \ref{prop:d2i}.}
    \label{fig:logistic}
\end{figure}

\begin{example}\label{ex:logistic}
We consider the example of Noisy Logistic map from \cite{Kostic2022}. Noisy logistic map, a non-linear dynamical system defined by the recursive relation $x_{t + 1} = (4x_{t}(1 - x_{t}) + \xi_{t}) \mod 1$ over the state space ${\cal X} = [0 , 1]$. Here, $\xi_{t}$ is i.i.d. additive {\em trigonometric} noise as defined in~\cite{Ostruszka2000}. The probability distribution of trigonometric noise is supported in $[-0.5,0.5]$ and is proportional to $\cos^{N}(\pi\xi)$, $N$ being an {\em even} integer. In this setting, the true invariant distribution, transition kernel and Koopman eigenvalues are easily computed, since the Koopman operator is isometrically isomorphic to $(N+1)\times (N+1)$ matrix. In Figure \ref{fig:logistic} (left) we show the true pseudospectrum with its true low rank approximation, while in the plot on the right we compare the true low rank localization area with the one obtained by learning from data using RRR estimator \eqref{eq:rrr} in the Gaussian RKHS by using Theorem \ref{thm:ps_rrr}.  
\end{example}

\begin{figure}[ht!]
    \centering
   \includegraphics[scale=0.45]{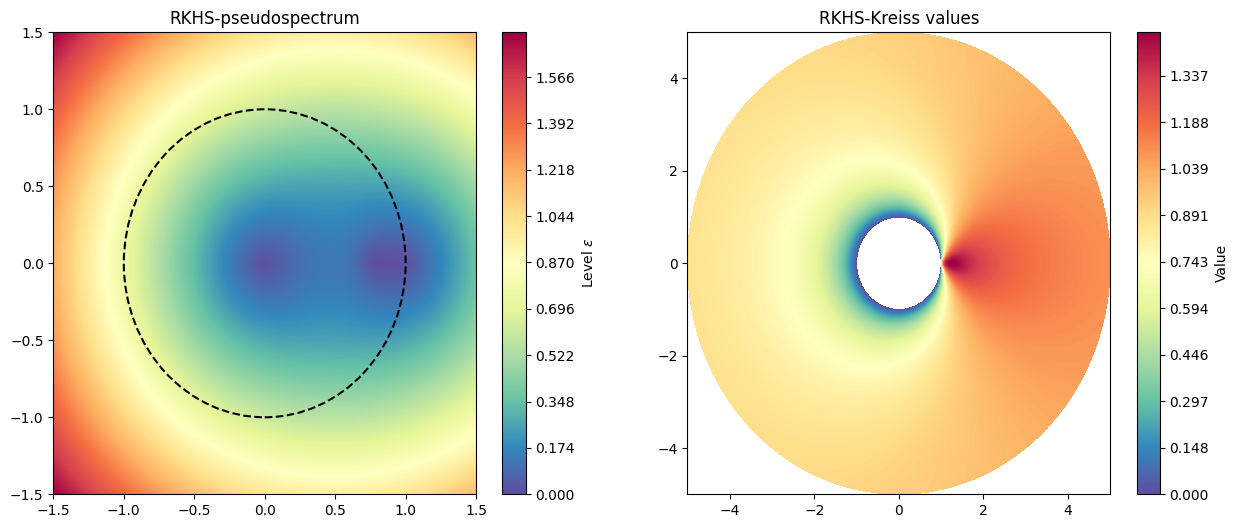}\\
    \includegraphics[scale=0.45]{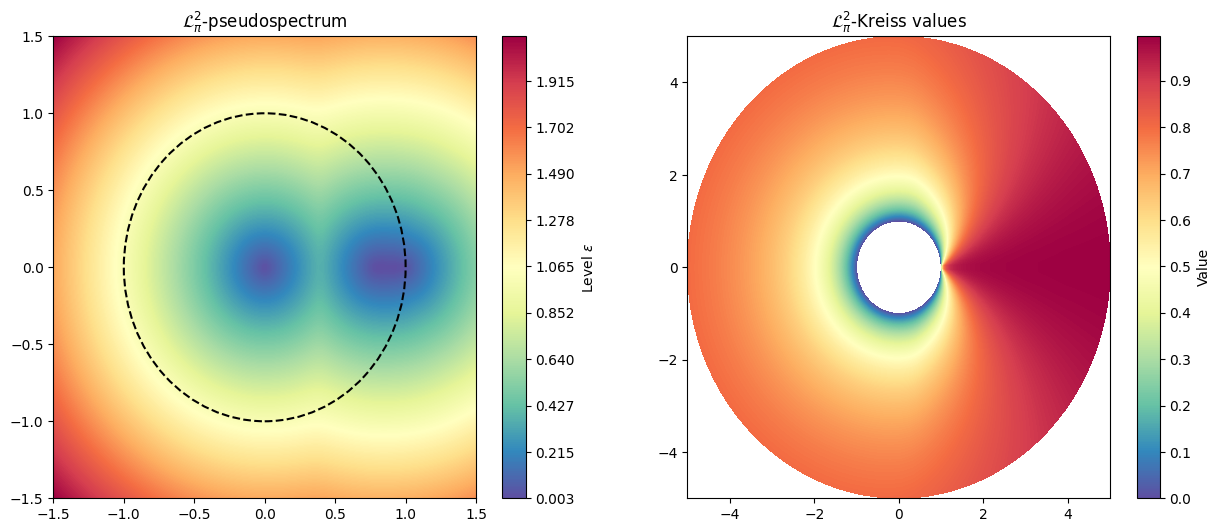}   
    \caption{Estimated pseudospectrum of the transfer operator for the normal Ornstein–Uhlenbeck process of Example \ref{ex:ou} via Theorem \ref{thm:ps_rrr}. The drift is a symmetric matrix with entries $a_{11}=a_{22}=-0.7$ and $a_{12}=a_{21}=0.3$, the RRR estimator \eqref{eq:rrr} is learned from $n=10^4$ samples, with rank $r= 20$, Tikhnov regularization $\reg=10^{-6}$ and the Gaussian kernel. In the top row we show the pseudospectrum and Kreiss constant of $\EEstim_{\reg,r}\colon\spH\to\spH$ computed via \eqref{eq:rkhs_lrps}, while in the bottom row we show the $\Lii$ estimation obtained by \eqref{eq:l2_lrps}. On the left we see pseudospectrum and on the right the values over which Kreiss constant is obtained as the maximum.}
    \label{fig:ou_normal}
\end{figure}

\begin{example}\label{ex:ou}
We consider the example of a 2D Ornstein–Uhlenbeck process given as the solution of the stochastic differential equation $dX_t = A\,dt+\sigma dW_t$, where $A\in\R^{2,2}$ is a drift matrix, $\sigma>0$ is a diffusion coefficient, and $dW_t$ is standard Brownian motion in $\R^2$. This models systems like the Vasicek interest rate, and neural dynamics, where fluctuations return to equilibrium. If the real parts of $A$'s eigenvalues are negative, the process has an invariant Gaussian distribution with covariance $\Sigma_\infty$ satisfying Lyapunov's equation: $A\Sigma_\infty + \Sigma_\infty A^\top = -\sigma^2 I$, that is $\pi \equiv {\cal N}(0,\Sigma_\infty)$. Sinnce diffusion term is $\sigma I$, the transfer operator $\cKoop\colon\Lii\to\Lii$ of this process is normal, if and only if the drift matrix is normal \cite{ross1995stochastic}. However, even for non-normal drift,  the transient behavior in $\Lii$-norm from perturbations of equilibrium is always bounded by $1$, irrespectivly of normality. Indeed, this is due to $\srad(\Koop)=\norm{\Koop}=1$, while $\srad(\cKoop)\leq \norm{\cKoop}<1$, which implies that  the $\Lii$-Kreiss constant is always $1$ for transfer operators.

In Figures \ref{fig:ou_normal} and \ref{fig:ou_nonnormal} we show the  pseudospectral estimation in the hypothesis domain $\spH$ and the true domain $\Lii$ for normal and non-normal  drifts, respectively. In both cases we use RRR estimator \eqref{eq:rrr} learned from $n=10^4$ samples, with rank $r= 20$, Tikhnov regularization $\reg=10^{-6}$ and the Gaussian kernel.  Importantly, we note that the fact that $\Lii$-Kreiss constant is always one for Transfer operators of time-homogeneous Markov processes can be used as a tool to cross-validate estimators. Indeed, we report that for taking the rank of the estimator to be 30 or choosing too large or to small Tikhonov regularization parameter results in $\Lii$-Kreiss constant much larger than one, indicating misaligned behavior of the estimator. 
\end{example}

\begin{figure}[ht!]
    \centering
   \includegraphics[scale=0.45]{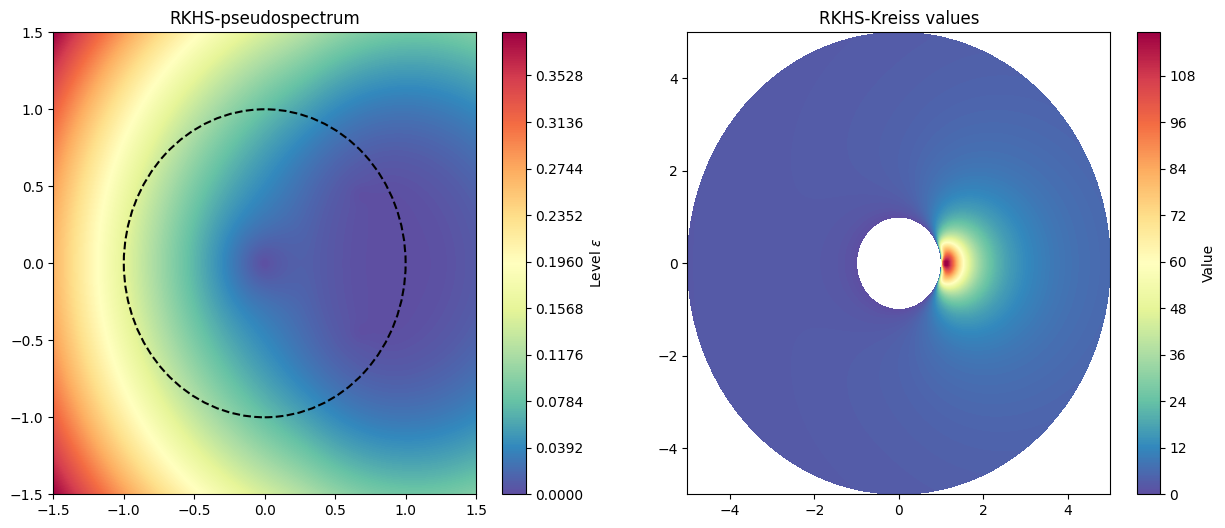}\\
    \includegraphics[scale=0.45]{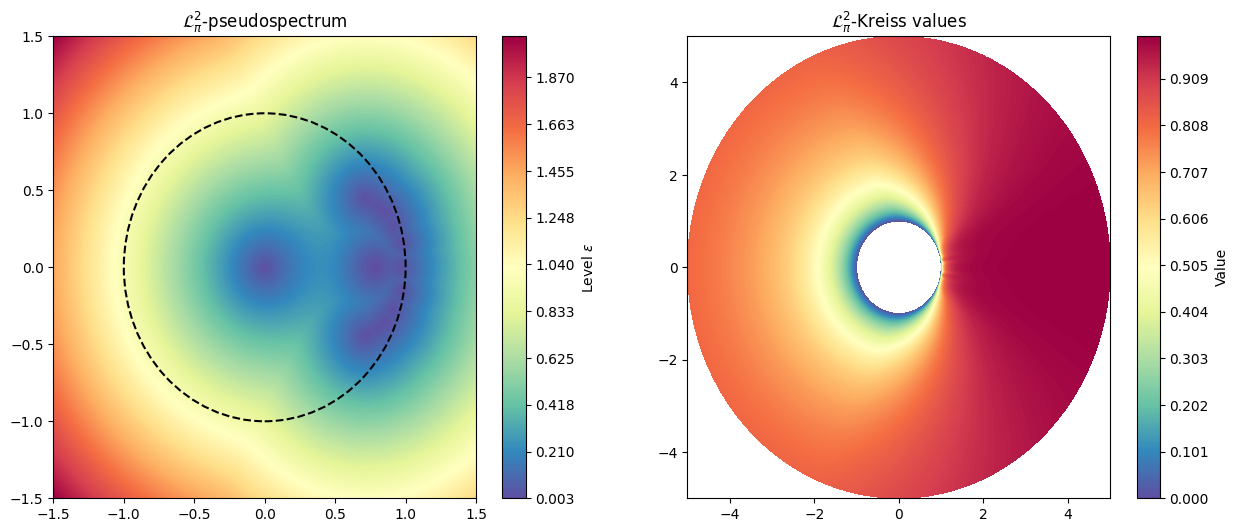}   
    \caption{Estimated pseudospectrum of the transfer operator for the nonnormal Ornstein–Uhlenbeck process of Example \ref{ex:ou} via Theorem \ref{thm:ps_rrr}. The drift is a non-normal matrix with entries $a_{11}=a_{22}=-0.7$, $a_{12}=100$ and $a_{21}=-0.1$. In the top row we show the pseudospectrum and Kreiss constant of $\EEstim_{\reg,r}\colon\spH\to\spH$ computed via \eqref{eq:rkhs_lrps}, while in the bottom row we show the $\Lii$ estimation obtained by \eqref{eq:l2_lrps}. On the left we see pseudospectrum and on the right the values over which Kreiss constant is obtained as the maximum.}
    \label{fig:ou_nonnormal}
\end{figure}

\section{Concluding remarks}

This paper has established a new computational paradigm for pseudospectral analysis by systematically exploiting low-rank structure. The central theoretical innovation is Theorem~1, which provides an exact characterization of the pseudospectrum of arbitrary low-rank matrices. This result transforms the expensive computation of resolvent norms - traditionally requiring $\mathcal{O}(d^3)$ operations per point in the complex plane - into an eigenvalue problem of dimension $2r \times 2r$, where $r \ll d$ is the rank. The subsequent Propositions~1 and~2 further enable efficient computation of key pseudospectral intersections with circles and lines, facilitating scalable algorithms for distance to instability and Kreiss constants.

Beyond exact low-rank operators, our framework provides principled approximations for general matrices.  Theorem~2 and Corollary~2 establish rigorous pseudospectral inclusion sets based on truncated and randomized low-rank approximations, with explicit error bounds linking approximation quality to pseudospectral accuracy. These results bridge randomized numerical linear algebra with spectral theory, offering a systematic trade-off between computational efficiency and precision. In practice, this enables pseudospectral analysis of matrices with dimensions where classical methods become infeasible, as demonstrated by the orders-of-magnitude speedups in our numerical experiments.

A particularly impactful application lies in data-driven dynamical systems.  Section~5 shows how our low-rank pseudospectral theory integrates naturally with modern operator learning techniques, enabling pseudospectral analysis of Koopman and transfer operators from trajectory data. Theorem~3 provides explicit formulas for computing pseudospectra of reduced-rank regression estimators, connecting statistical learning guarantees with dynamical systems analysis. This allows rigorous assessment of transient growth and stability margins in learned models of nonlinear and stochastic dynamics—a capability previously hindered by computational limitations.

Looking forward, several promising directions emerge. First, the structured eigenvalue problems in  Theorem~1 invite further algorithmic development, potentially leveraging recent advances in Hermitian eigenvalue solvers. Second, the extension to infinite-dimensional operators, suggested by our RKHS analysis, merits deeper theoretical investigation, particularly regarding developing optimal learning bounds. Third, applications to specific domains like fluid dynamics, where nonnormality is pronounced but system dimensions are large, present natural testbeds for our methodology. Finally, the integration with time-series forecasting and control design represents a practical frontier, where pseudospectral bounds could inform robust decision-making in data-driven settings.

In summary, this work transforms pseudospectral analysis from a computationally intensive tool for moderate-scale matrices to a scalable framework applicable to high-dimensional and data-driven systems. By unifying low-rank approximation theory with pseudospectral analysis, we provide both the theoretical foundations and practical algorithms to address long-standing computational barriers, opening new avenues for robust stability analysis in large-scale scientific and engineering applications.

\section{Acknowledgements}
The work of the third author has been supported by the Ministry of Science, Technological Development and Innovation (Contract No. 451-03-34/2026-03/200156) and the Faculty of Technical Sciences, University of Novi Sad through project “Scientific and Artistic Research Work of Researchers in Teaching and Associate Positions at the Faculty of Technical Sciences, University of Novi Sad 2026” (No. 01-3609/1). Also, the work of the first and third author has been supported by the The Provincial Secretariat for Higher Education and Scientific Research of Vojvodina, Research Grant 003870560 2025 09418 003 000 000 001 04 004.

\end{document}